\documentclass[11pt]{article}
\usepackage{amssymb,amsmath,amsthm,url}

\def\T{{\cal T}}
\def\P{{\cal P}}
\def\C{{\cal C}}
\def\A{{\cal A}}
\def\cS{{\cal S}}
\def\RSK{\operatorname{RSK}}
\def\la{\lambda}
\def\La{\Lambda}
\def\al{\alpha}
\def\be{\beta}

\def\eps{\varepsilon}
\def\si{\sigma}
\def\De{\Delta}

\def\sn{{\mathfrak S}_n}
\def\shape{\operatorname{sh}}

\def\Y{{\mathbb Y}}
\def\C{{\mathfrak C}}
\def\SW{\operatorname{SW}}
\def\Prob{\operatorname{Prob}}
\def\rev{\operatorname{rev}}
\def\evac{\operatorname{evac}}
\def\N{{\mathbb N}}
\def\tA{{\widetilde A}}
\def\m{{\cal M}}

\newtheorem{lemma}{Lemma}
\newtheorem{proposition}{Proposition}
\newtheorem{theorem}{Theorem}
\newtheorem{definition}{Definition}
\newtheorem{corollary}{Corollary}

\urldef\vershik\url{avershik@pdmi.ras.ru}
\urldef\natalia\url{natalia@pdmi.ras.ru}

\author{A.~M.~Vershik\thanks{%
St.~Petersburg Department of Steklov Institute of Mathematics, St.~Petersburg State University, Institute for Information Transmission Problems.
E-mail: \vershik.}
\and N.~V.~Tsilevich\thanks{%
St.~Petersburg Department of Steklov Institute of Mathematics.
E-mail: \natalia.}}
\title{The Schur--Weyl graph and Thoma's theorem\thanks{Supported by the RSF grant 21-11-00152.}}

\date{}

\begin{document}
\maketitle

\hfill{\it To the memory of S.~V.~Kerov {\rm(12.06.1946--30.07.2000)}}
\smallskip

\begin{abstract}
We define a graded graph, called the Schur--Weyl graph, which arises naturally when one considers simultaneously the RSK algorithm and the classical duality between representations of the symmetric and general linear groups. As one of the first applications of this graph, we give a new proof of the completeness of the list of discrete indecomposable characters of the infinite symmetric group.

{\bf Key words:} Schur--Weyl graph, RSK algorithm, Thoma theorem, central measures.
\end{abstract}

\section{Introduction}

In \cite{VK}, it was shown that the RSK algorithm generalized to infinite (Bernoulli) sequences allows one to enumerate all characters of the infinite symmetric group. However, this observation did not provide a new proof of the completeness of the list of indecomposable characters (i.e., of Thoma's theorem~\cite{Thoma}), but it revealed some profound properties of measures corresponding to characters. The so-called youngization homomorphism from the $L^2$ Bernoulli space onto the Hilbert space of functions on infinite Young tableaux endowed with the resulting central measure, first, is an isomorphism, as is proved in the very important papers~\cite{RS, Sniady}, and, second, allows one to give new realizations of some irreducible representations of the infinite symmetric group.

The finite RSK algorithm associates with a finite word in an ordered alphabet a pair of Young tableaux, a semistandard $P$-tableau and a standard $Q$-tableau. It is in this form that it is considered in most of the literature. The key element of our approach is a ``dynamic'' view (first suggested in~\cite{VK}) on the RSK algorithm, when it is applied to growing sequences of symbols and one considers the dynamics of the arising tableaux. The youngization homomorphism from~\cite{VK} uses the $Q$-tableau, which also just grows. However, the behavior of the $P$-tableau is much more complicated. Nevertheless, as this paper shows (see also~\cite{13}), tracking the dynamics of the $P$-tableau is also doable and very useful. It is in this way that we have led to the definition of the Schur--Weyl graph, which is a covering of the Young graph and reflects the dynamics of the $P$-tableau when the RSK algorithm is applied to a~growing sequence of symbols. Besides, this graph is related in a~natural way to the Schur--Weyl duality between representations of the general linear groups and symmetric groups. The relation between the RSK algorithm and the Schur--Weyl duality was  observed earlier, but the authors have not been able to find precise statements in the literature. It is clear that this duality must show itself also in the infinite version of the algorithm; further research is especially needed in this direction. We emphasize that in the general context of this paper, an important role is played by  the machinery of the theory of filtrations (decreasing sequences of measurable partitions), which was developed for the needs of ergodic theory, but turned out to be adequate for the problems under consideration (see, e.g.,~\cite{new}).

In this paper, we use the Schur--Weyl graph to obtain a new proof of Thoma's theorem~\cite{Thoma} in the case of discrete central measures. There are several proofs of this important theorem (\cite{VK81}, \cite{Ok}, \cite{KOO}, etc., see the survey~\cite{Three}); however, our argument is much simpler and of a purely combinatorial nature. Like many other proofs, it relies on the so-called ergodic method~\cite{V74}, but the direct application of this method to the Young graph entails the need for calculating the asymptotics of dimensions of skew Young diagrams; this is a difficult problem, which makes it necessary to develop special techniques, such as, e.g., shifted Schur functions. However, using the RSK algorithm and the Schur--Weyl graph allows us to reduce the problem under consideration to calculating the asymptotics of dimensions of ordinary Young diagrams, which is a much easier task. For another combinatorial argument for the same discrete case, see~\cite{14}. However, it is important to mention that the difficulty of the problem of describing the central measures is significantly different in the discrete and continuous (zero frequencies) cases.

The paper is organized as follows. In Section~\ref{sec:youngization}, we consider the youngization homomorphism from the space of infinite words to the space of infinite Young tableaux and its properties. In Section~\ref{sec:SW}, the Schur--Weyl graph is introduced and it is shown how the problem of describing the central measures on the Young graph can be reduced to the analogous problem for the Schur--Weyl graph. In Section~\ref{sec:ergrow}, we apply the ergodic method to solve this problem for the Schur--Weyl graph in the case of finitely many rows; the proof is split into lemmas, presented in Sections~\ref{s:densities}--\ref{s:calc},  and put together in Section~\ref{s:mainproof}. Finally, in Section~\ref{sec:gen} we explain how the obtained result can be extended to the case of central measures supported on tableaux with finitely many rows and columns.

The reader is assumed to be familiar with  the Robinson--Schensted--Knuth (RSK) algorithm  and its basic properties (see, e.g.,~\cite{Stanley, Fulton}); for  background on the representation theory of the infinite symmetric group (the Young graph, central measures, Thoma parameters, etc.), see, e.g.,~\cite{Kerov, BO}.

Throughout the paper, Young tableaux are filled with symbols from a~fixed alphabet~$\A$.

\section{Youngization}\label{sec:youngization}
Let $k\in\N$ and consider the alphabet $\A=\{1,2,\ldots,k\}$ consisting of $k$ ordered symbols. Consider the space  $X=\A^\infty$ of infinite words in the alphabet~$\A$.

By $\RSK(w)=(P(w),Q(w))$ we denote the result of applying the Ro\-bin\-son--Schensted--Knuth (RSK) correspondence  to a finite sequence (word)~$w$ in the alphabet~$\A$; thus, $P(w), Q(w)$ is a pair of Young tableaux of the same shape (with at most $k$~rows), which will be denoted by~$\shape(w)$; the tableau~$P(w)$ is semistandard, while the tableau~$Q(w)$ is standard. Given an infinite sequence
 $x\in X$, denote by $[x]_n\in \A^n$ its initial segment of length~$n$.

In \cite{VK}, the following map from the space of infinite sequences to the space~$\T(\Y)$ of infinite standard Young tableaux (or, which is the same, the set of infinite paths in the Young graph~$\Y$) was introduced.

\begin{definition}\label{def:youngization}
{\rm Apply the $\RSK$ algorithm to the initial segments~$[x]_n$ of a~sequence $x\in X$ and set $\RSK([x]_n)=(P_n(x),Q_n(x))$. It is clear from the construction of this algorithm that
$\lim\limits_{n\to\infty}Q_n(x)=Q(x)$ is an infinite standard Young tableau;  denote it by~$\pi(x)$.  The resulting map
$\pi:X\to\T(\Y)$
is called the {\it youngization}.}
\end{definition}

Endow $X$ with a Bernoulli measure $m_p^\infty$, where $p=(p_1,p_2,\ldots,p_k)$, $p_i=\Prob(i)$, and let $\m_{(p,0,0)}$ be the central measure on~$\T(\Y)$ with Thoma parameters~$(p,0,0)$ (note that the measure~$\m_{(p,0,0)}$ is supported on the subset of tableaux with at most $k$~rows). In~\cite{VK}, it is proved that
{\it the youngization is a homomorphism of measure spaces}
\begin{equation}\label{hom}
\pi:(X,m_p^\infty)\to(\T(\Y),\m_{(p,0,0)}).
\end{equation}

The following measurable partitions are defined in a natural way on the space of infinite Young tableaux~$\T(\Y)$:
the {\it cylinder} partition~$\xi_n$ of level~$n$, whose element is the set of infinite paths in the Young graph with fixed initial segment of length~$n$ and arbitrary ``tail'', and the {\it tail} partition~$\eta_n$ of level~$n$, whose element is the set of infinite paths in the Young graph with fixed $(n+1)$-tail and arbitrary beginning.
In~\cite{new}, we introduced the following partitions on the space~$X$.

\begin{definition}
{\rm The {\it Young cylinder partition} and the {\it Young tail partition} of the space~$X$ of infinite sequences are the partitions
 $\bar\xi_n:=\pi^{-1}\xi_n$ and $\bar\eta_n:=\pi^{-1}\eta_n$, respectively.}
\end{definition}

Clearly, the sequence of partitions~$\bar\xi_n$ is monotonically increasing, while the sequence of partitions~$\bar\eta_n$ (called the {\it Young filtration}) is monotonically decreasing. Recall (see, e.g.,~\cite{LLT}) that the {\it plactic class}~$\P_t$ (respectively,  the {\it coplactic class~$\C_t$}) corresponding to a given Young tableau~$t$ of size~$n$ is the set of all words~$u$ of length~$n$ such that $P(u)=t$ (respectively, $Q(u)=t)$. In~\cite{new}, it is shown that the Young partitions on the space~$X$ can be described as follows:
\begin{itemize}\itemsep=-1mm
\item the elements of~$\bar\xi_n$ are indexed by the standard Young tableaux~$t$ of size~$t$ and coincide with the coplactic classes~$\C_t$;
\item the elements of~$\bar\eta_n$ are indexed by the pairs~$(t,y)$, where $t$ is a semistandard Young tableaux of size~$n$ and $y\in X$, and have the form
$\{x\in X: [x]_n\in\P_t,\,(x_{n+1},x_{n+2},\ldots)=y\}$.
\end{itemize}
Also, in \cite{new} it is proved that
\begin{itemize}\itemsep=-1mm
\item the sequence of partitions $\bar\xi_n$ is total, i.e., converges in the weak topology to the partition~$\eps$ into singletons;
\item the sequence of partitions $\bar\eta_n$  is ergodic, i.e., converges in the weak topology to the trivial partition~$\nu$.
\end{itemize}

In particular, it follows that the youngization~\eqref{hom} is in fact an {\it isomorphism} of measure spaces, which gives (in this special case of finitely many rows) another proof of this result, originally proved in~\cite{Sniady}.

It is natural to pose the problem of describing all ergodic measures on the space~$X$ invariant with respect to the Young tail partitions. In the next section, we introduce the so-called Schur--Weyl graph, for which this problem is equivalent to the problem of describing all ergodic central measures.

\section{The Schur--Weyl graph}\label{sec:SW}

\begin{definition}\label{def:SW}
{\rm The {\it Schur--Weyl graph} $\SW_k$ is the graded graph defined as follows:
\begin{itemize}\itemsep=-1mm
\item the vertices of level~$n$ are all semistandard Young tableaux with $n$~cells filled with symbols from the alphabet~$\A$;

\item tableaux $t_1$ and $t_2$ are joined by an edge if $t_2$~is obtained from~$t_1$ by row inserting an arbitrary element~$i\in \A$.
\end{itemize}}
\end{definition}

Clearly, the graph $\SW_k$ has the following properties:
\begin{itemize}\itemsep=-1mm
\item every vertex has exactly $k$ outgoing edges;
\item the dimension of a vertex~$t$ (the number of paths  leading to this vertex from the initial vertex~$\emptyset$ of level~$0$) is~$\dim\la$, where $\la=\shape(t)$ is the shape of~$t$.
\end{itemize}

Note that the graph $\SW_k$ can be regarded as a covering of the Young graph~$\Y$ (more exactly, of the subgraph~$\Y_k$ in~$\Y$ consisting of the diagrams with at most $k$~rows): the fiber over a vertex (Young diagram)~$\la\in\Y$ consists of all semistandard tableaux of shape~$\la$. On the other hand, this construction is related in a natural way to the classical Schur--Weyl duality (see below).

\smallskip\noindent{\bf Example.}
For $k=2$, the Schur--Weyl graph~$\SW_2$ has the following form:
\begin{itemize}\itemsep=-1mm
\item the vertices of level~$n$ are all pairs of the form $(\la,r)$ where $\la=(\la_1,\la_2)$ is a Young diagram with $n$~cells and at most two rows and $r$ is an integer from the interval $0\le r\le\la_1-\la_2$.
Set $k(\la):=\la_1-\la_2$.
\item
$(\la,r)\nearrow(\mu,s)$ $\iff$ {\rm(i)}~$\la\nearrow \mu$ (i.e., $\mu$ is obtained from~$\la$ by adding one cell) and {\rm(ii)}~$s=r$ or $s=r+1$ if $k(\mu)=k(\la)+1$, and $s=r$ if $k(\mu)=k(\la)-1$.
\end{itemize}
\smallskip

\begin{lemma}\label{l:SW}
The set ${\cal T}_n(\SW_k)$ of all paths of length~$n$ \textup(respectively, the space~${\cal T}(\SW_k)$ of infinite paths\textup) in the Schur--Weyl graph  can be identified in a natural way with the set~$\A^n$ of words of length~$n$ \textup(the space $X=\A^\infty$ of infinite words\textup) in the alphabet~$\A$.
\end{lemma}

\begin{proof}
A sequence $u=(x_1,\ldots,x_n)\in \A^n$ is identified with the vertex $t=P(u)$ where $(P(u),Q(u))=\RSK(u)$.
\end{proof}

Thus, the space of functions on paths of length~$n$ in the graph~$\SW_k$ can be identified in a natural way with the space $(\mathbb C^k)^{\otimes n}$; then the subspace corresponding to the paths leading to tableaux with a given diagram~$\la$ (i.e., lying in the fiber over the vertex ${\la\in\Y_k}$ of the Young graph) gets identified with the irreducible, with respect to $\operatorname{GL}(k,\mathbb C)\otimes\sn$, component
 $M_\la\otimes V_\la$ of the Schur--Weyl decomposition, where $V_\la$ is the space of the irreducible representation of the symmetric group~$\sn$ corresponding to the diagram~$\la$ and $M_\la$  is the space of the irreducible representation of the group $\operatorname{GL}(k,\mathbb C)$ corresponding to this diagram. One may say that  $\sn$ acts on the ``base'', i.e., on the space spanned by the paths in the Young graph~$\Y_k$ leading to the given vertex~$\la$, while $\operatorname{GL}(k,\mathbb C)$ acts in the ``fibers'' of the graph~$\SW_k$ over each such vertex.

The description of Young partitions on the space~$X$ implies the following result.

\begin{proposition}\label{prop:equiv}
Under the identification from Lemma~{\rm\ref{l:SW}}, the Young tail partition~$\bar\eta_n$ on the space~$X$ is identified with the tail partition of level~$n$ on the space~${\cal T}(\SW_k)$ of infinite paths in the Schur--Weyl graph.
\end{proposition}

\begin{corollary}
Every \textup(ergodic\textup) central measure on the space~${\cal T}(\SW_k)$ of infinite paths in the Schur--Weyl graph~$\SW_k$ is the image, under the identification from Lemma~{\rm\ref{l:SW}}, of an \textup(ergodic\textup) measure on the space~$X$ invariant with respect to the Young filtration.
\end{corollary}

On the other hand, the following lemma shows that in order to find all central measures on the Young graph, it suffices to find all central measures on the Schur--Weyl graph.

\begin{lemma}\label{l:toThoma}
For every \textup(ergodic\textup) central measure~$\nu$ on the subgraph~$\Y_k$ of the Young graph there exists an \textup(ergodic\textup) central measure~$\mu$ on the Schur--Weyl graph~$\SW_k$ that projects to~$\nu$ under the youngization.
\end{lemma}

\begin{proof}
For each Young diagram~$\la$, consider the ``maximal'' tableau~$S_\la$ of shape~$\la$ in which each column of length~$i$ is filled with the consecutive integers $k-i+1,\ldots,k$. Clearly, the subgraph in~$\SW_k$ consisting of the maximal tableaux is isomorphic to~$\Y_k$. The image~$\mu$ of the measure~$\nu$ under this isomorphism is as required. Obviously, if $\nu$~is ergodic, then $\mu$~is also ergodic.
\end{proof}

To describe the central measures, we use the {\it ergodic method}~\cite{V74}, in its version for  central measures on the path space of a graded graph presented in~\cite{Kerov}. Namely, let $\Gamma$~be a graded graph and $\T(\Gamma)$~be the space of infinite paths on~$\Gamma$. Given vertices $v,w$ in~$\Gamma$, denote by~$\dim(v,w)$ the number of paths  from~$v$ to~$w$, and let $\dim v:=\dim(\emptyset,v)$ be the dimension of~$v$. Also, denote by~$C_v$ the cylinder set of paths passing through~$v$.

\begin{theorem}[ergodic method \cite{Kerov}]\label{th:ergmethod}
Let $M$ be an ergodic central measure on the space~$\T(\Gamma)$. Then for $M$-a.e.\
path $t=(v_0,v_1,\ldots, v_n,\ldots)$, for all vertices $v\in\Gamma$, the following limits exist:
$$
\lim_{n\to\infty}\frac{\dim(v,v_n)}{\dim v_n}=\phi(v).
$$
The cylinder distributions of the measure~$M$ are given by the formula
$$
M(C_v)=\dim v\cdot\phi(v).
$$
\end{theorem}

In other words, let $\C_n(t)$ be the class of the tail partition of level~$n$ containing the finite segment
$t_n=(v_0,\ldots,v_n)$ of an infinite path~$t$ and $\mu_t^{(n)}$~be the uniform measure on~$\C_n(t)$; then for $M$-a.e.\ path~$t$ the measures~$\mu_t^{(n)}$ weakly converge to~$M$.

\section{Realization of the ergodic method}\label{sec:ergrow}

Our purpose is to use the Schur--Weyl graph to prove Thoma's theorem, which gives the list of all ergodic central measures on the Young graph, in the case of measures supported on tableaux with finitely many rows and finitely many columns. To make the proof clearer, in this section we consider the case of finitely many rows, and then in Section~\ref{sec:gen} describe how the result can be extended to a more general case.

\begin{theorem}\label{th:Thoma}
All ergodic central measures on the subgraph~$\Y_k$ of the Young graph consisting of the diagrams with at most $k$~rows are exhausted by the Thoma measures~$\m_{(p,0,0)}$ where $p=(p_1,\ldots,p_k)$, $p_1\ge\ldots\ge p_k\ge0$, \break$p_1+\ldots+p_k=1$.
\end{theorem}

For convenience, the first two subsections of this section contain the lemmas used in the proof: in Section~\ref{s:densities} it is shown that, given a sequence $x\in X$, the normalized row lengths of the growing diagram~$\shape([x]_n)$ coincide with the limiting densities of elements in~$x$; Section~\ref{s:comb} contains combinatorial lemmas related to the properties of the RSK algorithm; and Section~\ref{s:calc} contains lemmas on the asymtpotics of dimensions of Young diagrams and on ``convolution'' of determinants.

\subsection{Densities of elements and normalized row lengths}\label{s:densities}

\begin{proposition}\label{prop:densgen}
Given $x\in X$, let $\shape([x]_n)=(\la_1^{(n)},\ldots,\la_k^{(n)})$. If the limiting densities of elements
\begin{equation}\label{dens}
p_j:=\lim_{n\to\infty}\frac{\#\{i\le n:x_i=j\}}{n},\quad j=1,\ldots,k,
\end{equation}
exist, then the normalized row lengths
\begin{equation}\label{rowdens}
\lim_{n\to\infty}\frac{\la_j^{(n)}}n=\hat p_j,\quad j=1,\ldots,k,
\end{equation}
exist too, where $\hat p_1,\ldots, \hat p_k$ are the densities $p_1,\ldots,p_k$ arranged in nonincreasing order: $\hat p_1\ge\ldots\ge \hat p_k$.
\end{proposition}

\begin{proof}
Induction on the size~$k$ of the alphabet~$\A$.

\begin{lemma}\label{l:density}
Proposition~{\rm\ref{prop:densgen}} holds for the alphabet $\A_2=\{1,2\}$.
\end{lemma}

\begin{proof}
It is convenient to regard a sequence from~$\A_2^n$ as a word $w=x_1\ldots x_n$ in the alphabet~$\A_2$.  Bracket every factor~$21$ in~$w$. The letters that are not bracketed form a subword~$w_1$ in~$w$. Bracket every factor~$21$ in~$w_1$. There remains a subword~$w_2$. Continue this procedure until we are left with a~word of the form $w_k=1^a2^b=x_{i_1}\ldots x_{i_{a+b}}$ with $a,b\ge0$. We say that the coordinates $x_{i_1},\ldots ,x_{i_{a+b}}$ of~$w$ are {\it free} and the other coordinates are {\it paired}. Then $\la_2^{(n)}$~is the number of pairs of paired coordinates in~$[x]_n$.

Note that if $p_1<p_2$, then for sufficiently large~$n$ all $1$s in~$[x]_n$ are paired, hence the second row grows with the same rate as the number of~$1$s, so the lemma in this case is obvious. Below we assume that $p_1\ge p_2$ (in particular, $p_2\le 1/2$).

Let $a_n:=\#\{i\le n:x_i=1\}$, $b_n:=\#\{i\le n:x_i=2\}$ and assume that  $\lim\limits_{n\to\infty}\frac{b_n}n=p_2$ (then, obviously, $\lim\limits_{n\to\infty}\frac{a_n}n=p_1$). Setting $r_n=\la_2^{(n)}$, we will prove that
$\lim\limits_{n\to\infty}\frac{r_n}n=p_2$. Assume to the contrary that (since, obviously,
${r_n\le b_n}$) there exists a subsequence $n_i\to\infty$ such that ${r_{n_i}<(p_2-\delta)n_i}$ with $\delta>0$. On the other hand, for sufficiently large~$i$ we have ${|a_{n_i}-p_1n_i|<t\delta n_i}$ and $|b_{n_i}-p_2n_i|<t\delta n_i$ by the definition of limit, where the coefficient ${t\in(0,1)}$ is to be chosen  later. Hence, the initial segment~$[x]_{n_i}$ contains $a_{n_i}-r_{n_i}>(p_1-p_2+(1-t)\delta)n_i$ free $1$s and ${b_{n_i}-r_{n_i}>(1-t)\delta n_i}$ free~$2$s, with all free $2$s being to the right of all free $1$s. Pick $m_i$  such that among the elements of~$[x]_{n_i}$ with indices larger than $m_i$ there are
$(1-t)\delta n_i$ free~$2$s (and no free $1$s). Then $m_i=n_i-(1-t)\delta n_i-2k_i$ where $k_i$ is the number of pairs of paired coordinates with indices larger than~$m_i$. Note that  $m_i> (p_1-p_2+(1-t)\delta)n_i\to\infty$.
Then
$$
\frac{b_{m_i}}{m_i}=\frac{b_{n_i}-(1-t)\delta n_i-k_i}{n_i-(1-t)\delta n_i-2k_i}=
\frac{\frac{b_{n_i}}{n_i}-(1-t)\delta-\frac{k_i}{n_i}} {1-(1-t)\delta-2\frac{k_i}{n_i}}<
\frac{p_2-(1-2t)\delta-u} {1-(1-t)\delta-2u},
$$
where  $u=\frac{k_i}{n_i}$.
For sufficiently small~$t$, the right-hand side decreases in~$u$, so it does not exceed the value at
$u=0$, i.e.,
$\frac{p_2-(1-2t)\delta} {1-(1-t)\delta}
<p_2-\eps$, where $\eps>0$ (for sufficiently small~$t$). Choosing such~$t$,
for sufficiently large~$i$ we have
$\frac{b_{m_i}}{m_i}\le p_2-\eps$,  contradicting the fact that $\lim\limits_{n\to\infty}\frac{b_n}n=p_2$.
\end{proof}

Thus, the induction base $k=2$ is proved. Now we prove the induction step
from~$k$ to~$k+1$.
Let $c_n:=\#\{i\le n:x_i=k+1\}$. Note that symbols~$k+1$ do not affect the growth of the subtableau in~$P([x]_n)$ consisting of $1,\ldots,k$. Hence, by the induction hypothesis, the row lengths
$\mu_1^{(n)},\ldots,\mu_k^{(n)}$ of this subtableau  satisfy the relations
$\lim\limits_{n\to\infty}\frac{\mu_j^{(n)}}{n-c_n}=\frac{q_j}{1-p_{k+1}}$ and, consequently,
$\lim\limits_{n\to\infty}\frac{\mu_j^{(n)}}n=q_j$, where $(q_1,\ldots,q_k)$ is the sequence
$(p_1,\ldots,p_k)$ arranged in nonincreasing order. Now, it follows from Lemma~\ref{l:density} that $\lim\limits_{n\to\infty}\frac{\la_1^{(n)}}n=\max(q_1,p_{k+1})=\hat p_1$ and the density of the elements~$k+1$ bumped into the second row is $\min(q_1,p_{k+1})$. Analogously, $\lim\limits_{n\to\infty}\frac{\la_2^{(n)}}n=\max(q_2,\min(q_1,p_{k+1}))=\hat p_2$, the density of the elements~$k+1$ bumped into the third row is  $\min(q_1,q_2,p_{k+1})$, and so on up to $\lim\limits_{n\to\infty}\frac{\la_2^{(k+1)}}n=\min(q_1,\ldots,q_k,p_{k+1})=\hat p_{k+1}$.
\end{proof}

\subsection{Combinatorial lemmas}\label{s:comb}

The {\it dual} Robinson--Schensted--Knuth algorithm RSK* differs from the ordinary one in that  an element~$i$ bumps the leftmost element {\it greater or equal to}~$i$. Let $\RSK^*(w)=(P^*(w),Q^*(w))$. Here $Q^*(w)$ is still a standard tableau, while
$P^*(w)$ is a {\it dual semistandard} tableau, i.e., a tableau obtained by transposing a semistandard one.

Let $w=w_1\ldots w_n\in\A^n$. Denote by~$\rev(w)$ the word obtained by reversing the order of letters in~$w$: $\rev(w)=w_n\ldots w_1$. The following relation between RSK and RSK* is well known (see, e.g., \cite{peal}):
\begin{equation}\label{duality}
\RSK(w)=(P,Q)\iff \RSK^*(\rev(w))=(P^t,\evac(Q^t)),
\end{equation}
where $^t$ stands for transpose and $\evac$ is the evacuation (Sch\"utzenberger involution).

It is also well known that the row insertion procedures underlying the $\RSK$ and $\RSK^*$ algorithms are invertible in the following sense. Let $w=w_1\ldots w_n\in\A^n$ and $w'=w_1\ldots w_{n-1}$. If $P(w)=T$ and we know the cell  added  at the last step of the algorithm (i.e., the diagram~$\shape(w')$), then  a~simple combinatorial procedure allows us to recover $P(w')$~and~$w_n$. Let $\square$ be a~corner cell of~$T$. Denote by $\tau(T,\square)$ and~$\eps(T,\square)$ (respectively, $\tau^*(T,\square)$ and $\eps^*(T,\square)$) the tableau and the element obtained in this way under the assumption that this cell has been added at the last step of the $\RSK$ (respectively, $\RSK^*$) algorithm.

Let $T$ be a semistandard tableau, $|T|=n$, $\shape(T)=\nu$. For $m\le n$, denote by $\cS_m(\nu)$ the set of tableaux of shape $\nu\setminus\mu$, where $|\mu|=m$, whose cells are filled with
integers $1,\ldots,m$ so that they decrease along each row and each column. Given $S\in\cS_m(\nu)$, we define a word $u=\Phi_T(S)\in\A^m$ as follows. Set ${T_0:=T}$ and then, by induction,
$$
u_j:=\eps^*(T_{j-1}^t,\square_j^t), \qquad T_j:=(\tau^*(T_{j-1}^t,\square_j^t))^t,\qquad j=1,\ldots,m,
$$
where $\square_j$ is the cell of~$S$ containing~$j$ and $\square^t_j$ is the corresponding cell  of~$S^t$. Also, denote by~$d(S)$ the dimension of the Young diagram obtained from~$\nu$ by deleting all cells of~$S$.

\begin{lemma}\label{l:focus}
Given $m\le n$ and $a\in\A^m$, set
$c_a(T):=\#\{y\in \A^n: {P(y)=T},\,\break [y]_m=a\}$. Then
$$
c_a(T)=\sum_{S\in\cS_m(\nu)\colon \Phi_T(S)=a}d(S).
$$
\end{lemma}

\begin{proof}
Given a word $w\in\A^n$, denote by~$[w]^m$ its final segment of length~$m$.
By~\eqref{duality} we have
$$
c_a(T)=\#\{z\in \A^n: P^*(z)=T^t, [z]^m=\rev(a)\},
$$
which implies the lemma.
\end{proof}

Given $u\in\A^n$, we define a word $\Psi(u):=u'_1\ldots u'_n$  as follows: $u'_i$ is the number of the column to which a cell has been added at the $i$th step of the algorithm~$\RSK^*(u)$.

\begin{lemma}\label{l:shape}
$Q(\rev(\Psi(u)))=Q(\rev(u))$. In particular,  $\shape(\rev([w]^m))=\shape(\rev([u]^m))$ for every $m\le n$.
\end{lemma}

\begin{proof}
Clearly, $\Psi(u)$ is a lattice word (i.e., every prefix of~$\Psi(u)$ contains at least as many $1$s as~$2$s, at least as many $2$s as~$3$s, etc.). Therefore, when calculating~$Q^*(\Psi(u))$, if the current element is~$j$, then a cell is added to the $j$th column, which means that
$Q^*(\Psi(u))=Q^*(u)$.
By~\eqref{duality}, it follows that $Q(\rev(\Psi(u)))=Q(\rev(u))$.
\end{proof}

Given $S\in\cS_m(\nu)$, let $w_S=w_1\ldots w_m\in\A^m$ where $w_i$ is the number of the row of~$S$ containing~$i$.

\begin{corollary}\label{cor:shape}
If $\Phi_T(S)=a$, then
$\shape(w_S)=\shape(a)$.
\end{corollary}
\begin{proof}
Note that
$\Phi_T(S)=a \iff T=P(b\rev(a))$ and $\Psi(b\rev(a))=v\rev(w_S)$ where $b,v\in\A^{n-m}$, and the claim follows from Lemma~\ref{l:shape}.
\end{proof}

\subsection{Other lemmas}\label{s:calc}

Let $\La=\La_n$ be a Young diagram growing in such a way that there exist normalized row lengths~\eqref{rowdens}. Consider its Frobenius notation: $\La=(\al_1,\ldots,\al_k|\beta_1,\ldots,\beta_k)$.
Let $a\in\A^m$, and let $m_i$ be the multiplicity of~$i$ in~$a$ (thus, $m_1+\ldots+m_k=m$).
Denote by~$\La^a$ the diagram obtained from~$\La$ by deleting $m_i$~cells from the $i$th row (if it exists).

Hereafter, by $C_a:=\{y\in X:[y]_m=a\}$ we denote the cylinder set of sequences that begin with~$a$. Recall that $m_{\hat p}(C_a)=\prod \hat p_i^{m_i}$ is the Bernoulli measure of this set.

\begin{lemma}\label{l:asymp}
As $n\to\infty$,
\begin{equation}\label{l:limits}
\frac{\dim\La^a}{\dim\La}\sim m_{\hat p}(C_a)\cdot\prod_{i<j}\frac{\al_i-\al_j-(m_i-m_j)}{\al_i-\al_j}.
\end{equation}
In particular, if $\la_i-\la_j\to\infty$ for all $i<j$, then
$$
\lim_{n\to\infty}\frac{\dim\La^a}{\dim\La}=m_{\hat p}(C_a).
$$
\end{lemma}

\begin{proof}
We use the formula for the dimension of a diagram in terms of its Frobenius notation (see~\cite[Proposition~1.3]{BO}):
$$
\dim\La=\frac{n!}{\prod_{i=1}^k\al_i!\beta_i!}\frac{\prod_{1\le i<j\le k}(\al_i-\al_j)(\beta_i-\beta_j)}{\prod_{i,j=1}^k(\al_i+\beta_j+1)}.
$$
Note that (for sufficiently large~$n$) for $\La^a=(\al'_1,\ldots,\al'_k|\beta'_1,\ldots,\beta'_k)$ we have $\al'_i=\al-m_i$, $\be'_i=\be_i$. Thus,
$$
\frac{\dim\La^a}{\dim\La}=\frac{(n-m)!}{n!}\prod_i\frac{\al_i!}{(\al_i-m_i)!}
\prod_{i<j}\frac{\al_i-\al_j-(m_i-m_j)}{\al_i-\al_j}\prod_{i,j}\frac{\al_i+\be_j+1}{\al_i-m_i+\be_j+1}.
$$
Since $\be_j\le k-1$ and $\frac{\al_i}{n}\to \hat p_i$ as $n\to\infty$,  the lemma follows.
\end{proof}

Denote by $\De(x)=\De(x_1,\ldots,x_n)$ the Vandermonde determinant in variables $x=(x_1,\ldots,x_n)$. Given an arbitrary Young diagram~$\mu$ with at most $n$~rows, set
$$
\De_\mu(x):=\sum_{\zeta}\De(x_1-\zeta_1,\ldots,x_n-\zeta_n),
$$
where the sum is over all different permutations~$\zeta$ of the sequence $\mu=(\mu_1,\mu_2,\ldots,\mu_n)$ (padded with $0$s  if necessary).

\begin{lemma}\label{l:vandermonde}
$\De_\mu(x)=n_\mu\De(x)$, where $n_\mu$ is the number of different permutations of~$\mu$.
\end{lemma}

\begin{proof}
It suffices to prove that
$$
\sum_{\si\in{\mathfrak S}_n}\De(x_1-\mu_{\si(1)},\ldots,x_n-\mu_{\si(n)})=n!\De(x),
$$
where ${\mathfrak S}_n$ is the symmetric group, since the sum in the left-hand side differs from $\De_\mu(x)$ exactly by the factor~$\frac{n!}{n_\mu}$. If the length~$\ell$ of~$\mu$  is equal to~$1$, this is a well-known identity for Vandermonde determinants (see, e.g., \cite[Proposition~7.192]{Grinberg}):
$$
\sum_{i=1}^n\De(x_1,\ldots, x_{i-1},x_i-t,x_{i+1},\ldots, x_n)=n\De(x).
$$
Applying it successively several times, one can easily prove the required identity by induction.
\end{proof}

\subsection{Proof of Thoma's theorem}\label{s:mainproof}

By Lemma~\ref{l:toThoma}, it suffices to prove that every ergodic central measure on the graph~$\SW_k$ projects to a Thoma measure.

We apply the ergodic method. Let $x\in X$, and assume that there exist densities~\eqref{dens}, and hence, by Proposition~\ref{prop:densgen}, normalized row lengths~\eqref{rowdens}. To simplify notation, we will assume that
$p_1\ge p_2\ge\ldots\ge p_k$; in the general case, the argument remains the same, with $p_i$~everywhere replaced by~$\hat p_i$.
Denote by~$\C_n(x)$ the class of the Young tail partition~$\bar\eta_n$ containing~$[x]_n$, and let $\mu_x^{(n)}$ be the uniform measure on~$\C_n(x)$.
Set  $T_n=P([x]_n)$ and $\La_n=\shape(T_n)$.

Fix $m\in\N$ and  $a\in\A^m$. We have
$$
\mu_x(C_a)=\lim_{n\to\infty}\mu_x^{(n)}(C_a)=\lim_{n\to\infty}\frac{c_a(x,n)}{\#\C_n(x)}=\lim_{n\to\infty}\frac{c_a(x,n)}{\dim\La_n},
$$
where $c_a(x,n)=\#\{y\in\C_n(x):[y]_m=a\}$.

Let $C_\la$ be the cylinder set of paths in the Young graph passing through a vertex~$\la$. It suffices to prove that for every Young diagram~$\la$ with $m$~cells,
\begin{equation}\label{nado}
\sum_{a\colon\shape(a)=\la}\mu_x(C_a)=\m_{(p,0,0)}(C_\la)=\dim\la\cdot s_\la(p),
\end{equation}
where $s_\la$ is a Schur function.

Given a word $w\in\A^m$, let $\La_n^w$ be the diagram obtained from~$\La_n$ by successively deleting a cell from the $w_1$th, $w_2$th, \ldots\ row (if  possible). In the notation of Lemma~\ref{l:asymp}, consider an arbitrary subsequence of indices~$n$ along which all differences~$\al_i-\al_j$ have a (finite or infinite) limit, and set
$$
\phi(w):=\lim_{n\to\infty}\frac{\dim\La^w_n}{\dim\La_n}.
$$
If cells cannot be deleted in this way, set $\phi(w)=0$. Also set $\Phi_{T_n}(w):=\Phi_{T_n}(S_w)$, where $S_w\in\cS_m(\La_n)$ is the tableau corresponding to this sequence of deleted cells. It follows from Lemma~\ref{l:focus} that
$$
\mu_x(C_a)=\sum_{w\colon\Phi_{T_n}(w)=a}\phi(w),
$$
and Corollary~\ref{cor:shape} of Lemma~\ref{l:shape} implies that
$$
\sum_{a\colon\shape(a)=\la}\mu_x(C_a)=\sum_{w\colon\shape(w)=\la}\phi(w).
$$

First, assume that
\begin{equation}\label{cond}
\la_i-\la_j\to \infty\quad\text{for any} \quad i<j
\end{equation}
(note that this condition is automatically satisfied for given~$i,j$ if  ${p_i\ne p_j}$). Then $\al_i-\al_j\to\infty$ for any $i<j$, and
by Lemma~\ref{l:asymp} we have
$\phi(w)=\prod_ip_i^{m_i(w)}$, where $m_i(w)$ is the multiplicity of~$i$ in~$w$. It follows that
$$
\sum_{a\colon\shape(a)=\la}\mu_x(C_a)=\sum_{w\colon\shape(w)=\la}\prod_i p_i^{m_i(w)}=\dim\la\cdot s_\la(p)
=\m_{(p,0,0)}(C_\la);
$$
here we have used the combinatorial definition of Schur functions
\begin{equation}\label{schur}
s_\la(p)=\sum_{S\colon\shape(S)=\la}\prod_ip_i^{m_i(S)},
\end{equation}
where  $m_i(S)$ is the multiplicity of~$i$ in a semistandard tableau~$S$, and the fact that for every tableau~$S$ of shape~$\la$ there are exactly $\dim\la$~words~$w$ such that $P(w)=S$. Thus, in the case under consideration equality~\eqref{nado} is proved.

Now assume that condition~\eqref{cond} is not satisfied. First, consider the case where it fails for all pairs~$i<j$, and hence ${p_1=\ldots=p_k=\frac1k}$. By Lemma~\ref{l:asymp},
$$
\phi(w)=\frac1{k^m}\cdot\prod_{i<j}\frac{\al_i-\al_j-(m_i-m_j)}{\al_i-\al_j}=
\frac1{k^m\De(\al)}\cdot\De(\al_1-m_1,\ldots,\al_k-m_k).
$$
Thus, the sum we are interested in is equal to
$$
\frac1{k^m\De(\al)}\sum_{w\colon\shape(w)=\la}\De(\al_1-m_1,\ldots,\al_k-m_k).
$$
Since the sum~\eqref{schur} defines a symmetric function, it follows that the last sum is a linear combination of determinants of the form~$\De_\mu(\al)$, and hence, by Lemma~\ref{l:vandermonde},
$$
\sum_{w\colon \shape(w)=\la}\phi(w)=\frac{\dim\la\cdot d_\la}{k^m},
$$
where $d_\la$ is the number of different semistandard tableaux of shape~$\la$. The right-hand side of the last equality is equal to  $\dim\la\cdot s_\la(\frac1k,\ldots,\frac1k)$, so equality~\eqref{nado} is proved in this case too.

Finally, consider the case where there are several families of equal densities. Namely, let
$q_1,\ldots,q_s$ be all different values of densities and $J_j:=\{i\colon p_i=q_j\}$, $j=1,\ldots,s$.  Set $n_j=\sum_{i\in J_j} m_i$.
By Lemma~\ref{l:asymp},
$$
\phi(w)=\prod_{j=1}^s\frac{q_j^{n_j}}{\De_j(\al)}\prod_{j=1}^s\De_j(\al-m),
$$
where $\De_j(\al)$ and $\De_j(\al-m)$ are the Vandermonde determinants in the variables~$\al_i$, $i\in J_j$,  and $\al_i-m_i$, $i\in J_j$, respectively. We are interested in the sum $\sum\limits_{w\colon\shape(w)=\la} \phi(w)$. Fix the block multiplicities $n_1,\ldots,n_s$ and consider the sum over all words~$w$ with these multiplicities. Similarly to the argument above, we conclude that, first, this sum splits into the product of sums over separate blocks and, second, the sum of the determinants~$\De_j(\al-m)$ in each block is equal to~$\De_j(\al)$ with an appropriate factor by Lemma~\ref{l:vandermonde}. The result is again~\eqref{nado}.

Thus, if $x\in X$ is a sequence such that the densities~\eqref{dens} exist, the weak limit~$\mu_x$ of the measures~$\mu_x^{(n)}$ projects to the Thoma measure~$\m_{(p,0,0)}$ under the projection of the Schur--Weyl graph onto the Young graph.  It follows (by considering convergent subsequences) that if the densities~\eqref{dens} do not exist, then the measures~$\mu_x^{(n)}$ have no limit.

The theorem is proved.

\smallskip\noindent{\bf Remarks.}
We say that a central measure~$\mu$ on the Schur--Weyl graph is {\it nondegenerate} if $\mu(C_a)\ne0$ for every finite word~$a$. One can easily see from the proof that a nondegenerate ergodic central measure with given densities on the Schur--Weyl graph is unique and coincides with the image of the corresponding Bernoulli measure under the isomorphism from Lemma~\ref{l:SW}. Apart from this measure, there exists a family of degenerate central measures, supported on tableaux with restrictions on the number of ``free'' elements (defined like in the proof of Lemma~\ref{l:density}). In particular, the measure constructed in Lemma~\ref{l:toThoma} is the ``most degenerate'' measure, supported on tableaux with no free elements.

\section{The general finite case}\label{sec:gen}

Let $\A_k:=\{1,\ldots,k\}$ and denote by
$\A^*_\ell:=\{i^*\colon i\in\A_\ell\}$ the alphabet {\it opposite} to~$\A_\ell$, in which the order  is defined as ${1^*>\ldots>\ell^*}$. Consider the alphabet $\tA=\A_k\cup\A^*_\ell$ with the order determined by the orders on $\A_k$~and~$\A^*_\ell$ and the convention that $a<b$ if $a\in\A_k$, $b\in\A^*_\ell$.
A Bernoulli measure $m_{\al,\be}$ on~$\tA$ is determined by vectors $\al=(\al_1,\ldots,\al_k)$ and $\be=(\be_1,\ldots,\be_\ell)$, where  $\al_i=\Prob(i)$, $\be_j=\Prob(j^*)$.

Following~\cite{VK}, consider the algorithm $\widetilde{\RSK}$ that acts on sequences
${x\in\tA^\infty}$ as follows: it applies the ordinary row insertion
to  symbols from~$\A$ (``row symbols'')  and the dual row insertion to symbols from~$\A^*$ (``column symbols''). Set
$\widetilde{\RSK}(x)=(\widetilde P(x),\widetilde Q(x))$. Here $\widetilde Q(x)$ is still a standard tableau, while
$\widetilde P(x)$ is a $(k,\ell)$-semistandard tableau (see~\cite{BR}).

As shown in~\cite{VK}, in this case the youngization (see Definition~\ref{def:youngization}) is again a homomorphism (and even, as proved in~\cite{Sniady}, an isomorphism) of measure spaces $\pi:(\tA^\infty,m_{\al,\be}^\infty)\to(\T(\Y),\m_{(\al,\be,0)})$, where $\m_{(\al,\be,0)}$ is the central measure on~$\T(\Y)$ with Thoma parameters~$(\al,\be,0)$.

In this case, we introduce the Schur--Weyl graph~$\SW$ similarly to Definition~\ref{def:SW}.  It is easy to see that Lemma~\ref{l:toThoma} remains valid (to obtain the ``maximal'' tableau with given diagram, we fill its first $k$~rows with elements of~$\A_k$ in the same way as before; the remaining part is of shape~$\nu^t$ where $\nu$ is a~diagram with at most $\ell$~rows, and we fill~$
\nu$ with elements of~$\A^*_\ell$ in a similar way and transpose). Hence, to prove Thoma's theorem in the general finite case, we can use the same approach as above: find the central measures on the Schur--Weyl graph by the ergodic method and prove that they project to Thoma measures. To reduce the proof of the general case to that considered above, we use the following lemma.

\begin{lemma}\label{l:genduality}
Given a word $w\in\tA^\infty$, denote by~$w^\dagger$ the word obtained from it by replacing each symbol~$i$ with~$i^*$ and each symbol~$j^*$ with~$j$.
Then $\widetilde Q(w^\dagger)=\widetilde Q(w)^t$.
\end{lemma}

\begin{proof}
It suffices to prove that $\shape(\widetilde Q(w^\dagger))=\shape(\widetilde Q(w))^t$ for every word ${w\in\tA^\infty}$.
Given $a,b\in\tA$, we write $a\nearrow b$ if $a<b$ or $a=b\in\A_k$, and $a\searrow b$ if $a>b$ or $a=b\in\A^*_\ell$. A word $u_1\ldots u_s\in\tA$ is said to be increasing (decreasing) if $u_i\nearrow u_{i+1}$ (respectively,  $u_i\searrow u_{i+1}$) for all $i=1,\ldots, s-1$. Clearly, $u$~ is increasing (decreasing) if and only if $u^\dagger$~is decreasing (increasing). Now the lemma follows from~\cite[Proposition~1]{VK}, according to which the length of the first row (column) of the diagram~$\shape(\widetilde Q(w))$ is the length of the largest increasing (decreasing) subsequence in~$w$, the sum of the lengths of the two first rows (columns) is the maximum cardinality of the union of two increasing (decreasing) subsequences in~$w$, etc.
\end{proof}

Given $x\in\tA^\infty$, denote by $\la_1'^{(n)},\ldots,\la_\ell'^{(n)}$ the column lengths of the diagram~$\shape([x]_n)$. In addition to the densities~\eqref{dens} and the normalized row lengths~\eqref{rowdens}, consider the densities
\begin{equation}\label{denscol}
q_j:=\lim_{n\to\infty}\frac{\#\{i\le n:x^*_i=j^*\}}{n},\quad j^*=1^*,\ldots,\ell^*,
\end{equation}
and the normalized column lengths
\begin{equation}\label{coldens}
\lim_{n\to\infty}\frac{\la_j'^{(n)}}n=\hat q_j,\quad j^*=1^*,\ldots,\ell^*.
\end{equation}

\begin{proposition}\label{prop:densities}
If there exist densities~\eqref{dens}~and~\eqref{denscol}, then there exist normalized row and column lengths \eqref{rowdens}~and~\eqref{coldens}, where $\hat p_1,\ldots, \hat p_k$ \textup(respectively, $\hat q_1,\ldots, \hat q_\ell$\textup) are the densities $p_1,\ldots, p_k$ \textup(respectively, $q_1,\ldots,q_\ell$\textup) arranged in nonincreasing order.
\end{proposition}

\begin{proof}
Assume that the densities~\eqref{dens}~and~\eqref{denscol} exist. Since elements from~$\A^*_\ell$ do not affect the dynamics of the subtableau consisting of elements from~$\A_k$, it follows from Proposition~\ref{prop:densgen} that $\varliminf\limits_{n\to\infty}\frac{\la_j^{(n)}}n\ge \hat p_j$.
Applying Lemma~\ref{l:genduality}, we see  that $\varliminf\limits_{n\to\infty}\frac{\la_j'^{(n)}}n\ge \hat q_j$. The remaining part of the proof is obvious.
\end{proof}

Then the proof follows the same scheme as for Theorem~\ref{th:Thoma}  with natural changes. Recall (see, e.g.,~\cite{VK81}) that the cylinder distributions of the Thoma measure~$\m_{(\al,\be,0)}$ are given by the formula
$$
\m_{(\al,\be,0)}(C_\la)=\dim\la\cdot\tilde s_\la(\al,\be),
$$
where $\tilde s_\la(\al,\be)$ are generalized (``hook'') Schur functions. For them, an analog of formula~\eqref{schur} holds in which the sum is over $(k,\ell)$-semistandard tableaux (see~\cite{BR}). It is not difficult to show that identity~\eqref{duality} takes the form
$$
\widetilde\RSK(w)=(P,Q)\iff \widetilde\RSK^*(\rev(w))=(P^t,\evac(Q^t)),
$$
where the $\widetilde\RSK^*$ algorithm applies the ordinary row insertion to column symbols and the dual one to row symbols. When calculating~$\Phi_T(S)$, at each step we apply
$\eps^*$~and~$\tau^*$ for row symbols and $\eps$~and~$\tau$ for column symbols; then Lemma~\ref{l:focus} remains without change. When calculating~$\Psi(u)$, we set $u'_i$~to be the number of the column (respectively, row) to which a cell is added at the $i$th step of the algorithm~$\widetilde\RSK^*(u)$ if $u_i\in\A_k$ (respectively, $u_i\in\A^*_\ell$); then Lemma~\ref{l:shape} (with $Q$ replaced by $\widetilde Q$) and Corollary~\ref{cor:shape} remain valid. In the asymptotics from Lemma~\ref{l:asymp}, similar factors corresponding to columns appear, and the remaining part of the proof goes without significant changes.

Thus, we obtain Thoma's theorem for the general finite case.

\begin{theorem}\label{th:genThoma}
All ergodic central measures on the Young graph supported on tableaux with at most $k$~rows and at most $\ell$~columns are exhausted by the Thoma measures~$\mu_{(\al,\be,0)}$.
\end{theorem}


\begin{thebibliography}{55}

\bibitem{peal}P.~Alexandersson, The symmetric functions catalog, \url{https://www2.math.upenn.edu/~peal/polynomials/tableauOperators.htm#RSKvsDualRSK}.

\bibitem{BR}A.~Berele and A.~Regev, Hook-Young diagrams with applications to combinatorics and to representations of Lie superalgebras, {\it Adv. in Math.} {\bf64}, 118--175 (1985).

\bibitem{BO} A.~Borodin and G.~Olshanski, {\it Representations of the Infinite Symmetric Group}, Cambridge University Press, 2017.

\bibitem{Fulton} W.~Fulton, {\it Young Tableaux With Applications to Representation Theory and Geometry},
Cambridge University Press, 1977.

\bibitem{Grinberg}D.~Grinberg, Notes on the combinatorial
fundamentals of algebra, {\tt 	arXiv:2008.09862}.

\bibitem{Kerov}  S.~V.~Kerov, {\it Asymptotic Representation Theory of the Symmetric Group and its Applications in Analysis}, Amer. Math. Soc., Providence, RI, 2003.

\bibitem{KOO}S.~Kerov, A.~Okounkov, and G.~Olshanski, The boundary of the Young graph with Jack edge multiplicities, {\it Int. Math. Res. Not.} {\bf 1998}, No.~4, 173--199 (1998).

\bibitem{VK}S.~V.~Kerov and A.~M.~Vershik, The characters of the infinite symmetric
group and probability properties of the Robinson--Schensted--Knuth algorithm,
{\it SIAM J. Algebraic Discrete Methods} {\bf7}, No.~1, 116–124 (1986).

\bibitem{LLT}A.~Lascoux, B.~Leclerc, and J.-Y. Thibon,
The plactic monoid, in:
M.~Lothaire (ed.), {\it Algebraic Combinatorics on Words}, Cambridge University Press, Cambridge (2002),
Chapter 6.

\bibitem{Ok}A.~Yu.~Okounkov, Thoma's theorem and representations of the infinite bisymmetric group, {\it Funct. Anal. Appl.} {\bf 28}, No.~2,  100--107  (1994).

\bibitem{RS}D.~Romik and P.~\'Sniady, Jeu de taquin dynamics on infinite Young tableaux
and second class particles, {\it Ann. Probab.} {\bf 43}, No.~2, 682--737 (2015).

\bibitem{Sniady}P.~\'Sniady, Robinson--Schensted--Knuth algorithm, jeu de taquin, and Kerov--Vershik measures on infinite tableaux,
{\it SIAM J. Discrete Math.} {\bf 28}, No.~2, 598--630 (2014).

\bibitem{Stanley} R.~Stanley, {\it Enumerative Combinatorics}, Vol.~2, Cambridge University Press, 1999.

\bibitem{Thoma}E.~Thoma,  Die unzerlegbaren, positiv-definiten Klassenfunktionen der abz\"ahlbar unendlichen, symmetrischen Gruppe, {\it Math. Z.} {\bf85}, 40--61 (1964).

\bibitem{V74}A.~M.~Vershik, Description of invariant measures for the actions of some infinite-dimensional groups,
{\it Sov. Math. Dokl.} {\bf15}, No.~4, 1396--1400 (1974).

\bibitem{Three}A.~M.~Vershik, Three theorems on the uniqueness of the Plancherel measure from different viewpoints, {\it Proc. Steklov Inst. Math.} {\bf 305}, 63--77 (2019).

\bibitem{13} A.~M.~Vershik, Combinatorial encoding of Bernoulli schemes and the asymptotic behavior of Young tableaux,  {\it Funct. Anal. Appl.} {\bf 54}, 77--92 (2020).

\bibitem{V21} A.~M.~Vershik, A method of defining central and Gibbs measures and the ergodic method, {\it Dokl. RAN} {\bf 497}, 30--34 (2021).

\bibitem{14} A.~M.~Vershik, Groups generated by involutions, numbering of posets, and central measures, {\it Uspekhi Mat. Nauk} {\bf 76}, No.~5(461) (2021).

\bibitem{VK81} A.~M.~Vershik and S.~V.~Kerov, Asymptotic theory of characters of the symmetric group, {\it Funct. Anal. Appl.}  {\bf15}, No.~4, 246--255 (1981).

\bibitem{new} A.~M.~Vershik and N.~V.~Tsilevich, Ergodicity and totality of partitions associated with the RSK correspondence, {\it Funct. Anal. Appl.} {\bf 55}, No.~1 (2021).

\end{thebibliography}
\end{document}